\documentclass{amsart}
\usepackage{amsfonts,amssymb,amscd,amsmath,enumerate,verbatim,newlfont,calc}
\usepackage{graphicx}
\usepackage{mathtools}
\usepackage{blkarray} 
\newcommand{\matindex}[1]{\mbox{\scriptsize#1}} 

\newtheorem{theorem}{Theorem}[section]
\newtheorem{theorem A}{Theorem A}
\newtheorem{theorem B}{Theorem B}
\newtheorem{lemma}[theorem]{Lemma}

\newtheorem{definition}[theorem]{Definition}

\newtheorem{proposition}[theorem]{Proposition}

\begin{document}
\address{Department of Mathematics, Ferdowsi University of Mashhad, Mashhad, Iran}
\title{The Bogomolov multiplier of groups of order $p^7$ and exponent $p$}
\author[Z. Araghi Rostami]{Zeinab Araghi Rostami}
\author[M. Parvizi]{Mohsen Parvizi}
\author[P. Niroomand]{Peyman Niroomand}
\address{Department of Pure Mathematics\\
Ferdowsi University of Mashhad, Mashhad, Iran}
\email{araghirostami@gmail.com, zeinabaraghirostami@stu.um.ac.ir}
\address{Department of Pure Mathematics\\
Ferdowsi University of Mashhad, Mashhad, Iran}
\email{parvizi@um.ac.ir}
\address{School of Mathematics and Computer Science\\
Damghan University, Damghan, Iran}
\email{niroomand@du.ac.ir, p$\_$niroomand@yahoo.com}
\address{Department of Mathematics, Ferdowsi University of Mashhad, Mashhad, Iran}
\keywords{Commutativity-Preserving exterior product, ${\tilde{B_0}}$-pairing, Curly exterior square, Bogomolov multiplier.}
\maketitle
\begin{abstract}
In this paper, the Bogomolov multiplier of $p$-groups of order $p^7$ ($p>2$) and exponent $p$ is given.

\end{abstract}
\section{\bf{Introduction}}
Let $K$ be a field, $G$ be a finite group, and $V$ be a faithful representation of $G$ over $K$. Then there is a natural action of $G$ upon the field of rational functions $K(V)$. Noether's problem which is introduced by Emmy Noether in \cite{25'}, is one of the important problems of invariant theory. It asks whether the field of $G$-invariant functions ${K(V)}^G$ is rational over $K$, or equivalently whether there exist independent variables $x_1,\ldots,x_r$ such that ${K(V)}^{G}(x_1,\ldots,x_r)$ becomes a pure transcendental extension of $K$. This problem has a close connection to Luroth's problem \cite{27'} and the inverse Galois problem \cite{28,28'}. Saltman in \cite{28} found examples of groups of order $p^9$ with a negative answer to Noether's problem, even when taking $K=\Bbb{C}$. His main method was the application of the unramified cohomology group ${H_{nr}^{2}}({\Bbb{C}(V)}^{G},{\Bbb{Q}}/{\Bbb{Z}})$ as an obstruction. This invariant had been used by Artin and Mumford \cite{1'}. They constructed unirational varieties over $K$ that were not rational. Bogomolov in \cite{2} proved that ${H_{nr}^{2}}({\Bbb{C}(V)}^{G},{\Bbb{Q}}/{\Bbb{Z}})$ is canonically isomorphic to
$$B_0(G)=\bigcap_{A\leq G}\  \ker \{{{res}_{A}^{G}}: H^2(G,{\Bbb{Q}}/{\Bbb{Z}})\longrightarrow H^2(A,{\Bbb{Q}}/{\Bbb{Z}})\},$$
where ${res}_{A}^{G}$ is the usual cohomological restriction maps and $A$ is an abelian subgroup of $G$. The group $B_0(G)$ is a subgroup of the Schur multiplier $\mathcal{M}(G)$ and Kunyavskii in \cite{16} named it the \emph{Bogomolov Multiplier} of $G$. Thus non triviality of the Bogomolov multiplier leads to counter-examples to Noether's problem. But it's not always easy to calculate the Bogomolov multiplier of groups.  Moravec in \cite{22} introduced an equivalent definition of the Bogomolov multiplier that has an important role in studying the Noether's problem. In this sense, he showed that if $G$ is a finite group, then $B_0(G)$ is naturally isomorphic to $\text{Hom} (\tilde{B_0}(G),{\Bbb{Q}}/{\Bbb{Z}})$, where the group $\tilde{B_0}(G)$ can be described as a section of the non abelian exterior square of a group $G$. Also, he proved that $\tilde{B_0}(G)=\mathcal{M}(G)/\mathcal{M}_0(G)$, in which the Schur multiplier $\mathcal{M}(G)$ or the same $H^2(G,{\Bbb{Q}}/{\Bbb{Z}})$ interpreted as the kernel of the commutator homomorphism $G\wedge G \rightarrow [G,G]$ given by $x\wedge y \rightarrow  [x,y]$, and $\mathcal{M}_0(G)$ is a subgroup of $\mathcal{M}(G)$ defined as $\mathcal{M}_0(G)=<x\wedge y \ | \  [x,y]=0 , \  x,y\in G>$. Therefore in the class of finite groups, $\tilde{B_0}(G)$ is naturally isomorphic to $B_0(G)$. Similar to the Schur multiplier, the Bogomolov multiplier can be explained as a measure of the extent to which relations among commutators in a group fail to be consequences of universal relation. Furthermore, Moravec's method relates the Bogomolov multiplier to the concept of commuting probability of a group and shows that the Bogomolov multiplier plays an important role in commutativity preserving central extensions of groups, that are famous cases in $K$-theory, (see \cite{12,22} for more information). There are some papers to compute this multiplier for some groups, for example Chu and Kang in \cite{4}, proved that the Bogomolov multiplier of all groups of order at most $p^4$, is trivial. Bogomolov in \cite{2} claimed that if $G$ is a group of order $p^5$, then ${\tilde{B_0}(G)}$ is trivial. This claim in \cite{6} was confirmed by Chu et al. for $p=2$.  Later, Moravec proved that this claim is false for $p=3$, and he showed that there are precisely three groups of order $3^5$ with the non trivial Bogomolov multipliers. Bogomolov in \cite{2} found some examples of groups of order $p^6$ with non trivial Bogomolov multiplier. Recently, Yin Chen and Rui Ma in \cite{3}, computed the Bogomolov multiplier for some $p$-groups of order $p^6$. On the other hand, in \cite{10,23}, Hoshi et al. and Moravec used different methods to prove that if $p>3$ is  prime and $G$ is a group of order $p^5$, then the Bogomolov multiplier is non trivial if and only if $G$ belongs to the isoclinism family ${\Phi}_{10}$. Also, Moravec in \cite{24} showed that for two isoclinic $p$-groups $G_1$ and $G_2$, $\tilde{B_0}(G_1)$ is isomorphic to $\tilde{B_0}(G_2)$. Thus if $\Phi$ denotes an isoclinism family of finite $p$-groups, so the Bogomolov multiplier of $\Phi$ is well defined and denotes by $\tilde{B_0}(\Phi)$. Furthermore, in \cite{6} Chue et al. classified all non abelian groups of order $2^5$ and $2^6$ with non trivial Bogomolov multipliers. Michilov in \cite{19}, computed $\tilde{B_0}(G)$ for some $p$-groups of nilpotency class $2$ that do not have the ABC (Abelian-By-Cyclic) property. Also, the Bogomolov multiplier was computed for groups of order $2^7$ by Jezernik and Moravec in \cite{12}. The classification of $p$-groups of order $p^7$ $(p>2)$ and exponent $p$, based on isomorphism, was conducted by Wilkinson in \cite{31}. In this paper, we determine which one of those groups has trivial Bogomolov multiplier.
\section{\bf{Preliminaries and notations}}
For the convenience of the reader, we give some known results and notations that will be used through the paper.
\\
Let $G$ be a group. The exterior square of $G$, $G\wedge G$, is defined to be the group generated by the symbols $x\wedge y$, subject to the following relations:
\renewcommand {\labelenumi}{(\roman{enumi})} 
\begin{enumerate}
\item{$xy\wedge z=(x^{y}\wedge z^{y})(y\wedge z),$}
\item{$x\wedge yz=({x}\wedge {z})(x^z\wedge y^z),$}
\item{$x\wedge x=1,$}
\end{enumerate}
for all $x,y,z\in G$.
\\
Let $L$ be a group. A function $\Phi: G\times G \to L$ is called an exterior pairing, if 
\renewcommand {\labelenumi}{(\roman{enumi})} 
\begin{enumerate}
\item{$\Phi(xy , z)=\Phi(x^{y} , z^{y})\Phi(y , z),$}
\item{$\Phi(x , yz)=\Phi({x} , {z})\Phi(x^z , y^z),$}
\item{$\Phi(x , x)=1,$}
\end{enumerate}
for all $x,y,z \in G$.
\\
It is easy to see that, an exterior pairing $\Phi$ determines a unique homomorphism of groups $\Phi ^{*} : G\wedge G \to L$ given by $\Phi ^{*} (x\wedge y)=\Phi(x,y)$ for all $x,y\in G$. An example of an exterior pairing is the commutator map $\kappa: G\times G \to [G,G]$ given by $(x,y) \longrightarrow [x,y]$ for all $x,y \in G$. It induces a homomorphism ${\tilde{\kappa}}: G\wedge G \to [G,G]$, given by $x\wedge y \longmapsto [x,y]$ for all $x,y \in G$. Miller in \cite{20} showed that the kernel of $\tilde{\kappa}$ is isomorphic to the Schur multiplier of $G$, that is
$$\mathcal{M}(G) = \{ \prod _{i=0}^{n} (x_i \wedge y_i)^{\varepsilon_i} \in G \wedge G  \  \vert \   {\varepsilon}_i = \pm 1 , \prod _{i=0}^{n} [x_i , y_i]^{\varepsilon_i} = 1 \ , \ i\in {\Bbb{N}\cup \{0\}}\}.$$
Moreover, $\mathcal{M}_0(G)$ is the subgroup of $\mathcal{M}(G)$ generated by all elements $x \wedge y$, such that $x,y \in G$ commute, that is
$$\mathcal{M}_0(G) = \{ \prod _{i=0}^{n} (x_i \wedge y_i)^{\varepsilon_i} \in G \wedge G  \  \vert \   {\varepsilon}_i = \pm 1 , [x_i , y_i] = 1  \ , \ i\in {\Bbb{N}\cup \{0\}}\}.$$
Moravec in \cite{22} proved that for a finite group $G$, ${B_0}(G)\cong \tilde{B_0}(G)={\mathcal{M}(G)}/{\mathcal{M}_0(G)}$. \\
Blyth and Morse in \cite{1}, introduced a suitable way for obtaining the exterior square of $G$. The advantage of this description is the ability of using the full power of the commutator calculus instead of computing with elements of $G\wedge G$.
\begin{definition}\cite{23}\label{d2.1}
Let $G$ be a group and let $G^{\varphi}$ be an isomorphic copy of $G$ via the mapping $\varphi: x\to x^{\varphi}$, for all $x\in G$. We define the group $\tau(G)$ as follows:
$$\tau(G) = <G , G^{\varphi} \ |\  [x,y^{\varphi}]^{z}=[x^{z},(y^z)^{\varphi}]=[x,y^{\varphi}]^{z^{\varphi}} , [x,x^{\varphi}]=1 \ ,\  x,y,z\in G>.$$
\end{definition}
Note that $G$ and $G^{\varphi}$ may be embedded into $\tau(G)$. So the groups $G$ and $G^{\varphi}$ can be viewed as subgroups of $\tau(G)$. 
Let $[G,G^{\varphi}]=<[x,y^{\varphi}] \ |\  x,y\in G>$. Now we give some properties of $\tau(G)$ and $[G,G^{\varphi}]$ that will be used frequently in our calculations.
\begin{proposition}\cite[Proposition 16]{1}\label{p2.2}
Let $G$ be a group. Then the map $\Phi: G\wedge G \to [G,G^{\varphi}]$ given by $(x\wedge y) \to [x,y^{\varphi}]$ for all $x,y\in G$, is isomorphism.
\end{proposition}
Let ${\kappa}^{*}=\tilde{\kappa}{\Phi}^{-1}$ be the composite map from $[G,G^{\varphi}]$ to $[G,G]$, $\mathcal{M}^{*}(G)=\ker{{\kappa}^{*}}$ and ${\mathcal{M}_0}^{*}(G)=\Phi(\mathcal{M}_0(G))$. More precisely, we have
$$\mathcal{M}^{*}(G) = \{ \prod _{i=0}^{n} [x_i , {y_i}^{\varphi}]^{\varepsilon_i} \in [G , G^{\varphi}] \   \vert \   {\varepsilon}_i = \pm 1 , \prod _{i=0}^{n} [x_i , y_i]^{\varepsilon_i} = 1   \ , \ i\in {\Bbb{N}\cup \{0\}}\}$$
and
$${\mathcal{M}_0}^{*}(G) = \{ \prod _{i=0}^{n} [x_i , {y_i}^{\varphi}]^{\varepsilon_i} \in [G , G^{\varphi}] \   \vert \   {\varepsilon}_i = \pm 1 ,  [x_i , y_i]= 1   \ , \ i\in {\Bbb{N}\cup \{0\}}\}.$$
It's an immediate consequence that in finite case, ${B_0}(G)$ is isomorphic to $\dfrac{\mathcal{M}^{*}(G)}{{\mathcal{M}_0}^{*}(G)}=\tilde{B_0}(G)$ (see \cite{22}). So in order to show that the triviality of $\tilde{B_0}(G)=0$, it suffices to show that $\mathcal{M}^{*}(G) \subseteq {\mathcal{M}_0}^{*}(G)$.
\\
In the following, we introduce farther properties of $\tau(G)$ and $[G,G^{\varphi}]$ that will be useful.
\begin{lemma}\cite[Lemmas 9, 10, 11]{1}\label{l2.3} 
Let $G$ be a group. The following statements, for all $x,y,z,v,w\in G$ and all $n,m\in \Bbb{N}$, hold:
\end{lemma}
\renewcommand {\labelenumi}{(\roman{enumi})}
\begin{enumerate}

\item{$[x,yz]=[x,z][x,y][x,y,z]$ and $[xy,z]=[x,z][x,z,y][y,z]$.}
\item{If $G$ is nilpotent of class $c$, then $\tau(G)$ is nilpotent of class at most $c+1$.}
\item{If $G$ is nilpotent of class $\leq 2$, then $[G,G^{\varphi}]$ is abelian.}
\item{$[x,y^{\varphi}]=[x^{\varphi},y]$.}
\item{$[x,y,z^{\varphi}]=[x,y^{\varphi},z]=[x^{\varphi},y,z]=[x^{\varphi},y^{\varphi},z]=[x^{\varphi},y,z^{\varphi}]=[x,y^{\varphi},z^{\varphi}]$.}
\item{$[[x,y^{\varphi}],[v,w^{\varphi}]]=[[x,y],[v,w]^{\varphi}]$.}
\item{$[x^n,y^{\varphi}]=[x,y^{\varphi}]^n=[x,(y^{\varphi})^n]$, where $[x,y]=1$.}
\item{If $[G,G]$ is nilpotent of class $c$, then $[G,G^{\varphi}]$ is nilpotent of class $c$ or $c+1$.}
\item{If $x$ and $y$ are commuting elements of $G$ of orders $m$ and $n$, respectively, then the order of $[x,y^{\varphi}]$ divides $\gcd(m,n)$.}
\end{enumerate}
In the process of the proofs in the third section, we need the following lemma to expanding the commutators.
\begin{lemma}\label{l2.4}
Let $G$ be a nilpotent group of class at most $6$. Then
\begin{align*}[x^n,y]=&[x,y]^n[x,y,x]^{\binom{n}{2}}[x,y,x,x]^{\binom{n}{3}}[x,y,x,x,x]^{\binom{n}{4}}[x,y,x,[x,y]]^{a(n)}\\&\quad[x,y,x,x,x,x]^{\binom{n}{5}}[x,y,x,[x,y],x]^{\binom{n}{3}+2{\binom{n}{4}}}[x,y,x,x,[x,y]]^{\binom{n}{3}+\binom{n}{4}},\end{align*}
for all $x,y\in \tau(G)$ and every positive integer $n$, where $a(n)=n(n-1)(2n-1)/6$.
\end{lemma}
\begin{proof}
We proceed by induction on $n$. The case that $n=1$ is obvious. By using the induction hypothesis we have:
\begin{align*}
[x^{n+1},y]&=[x^n,y][x^n,y,x][x,y]\\&=[x^n,y][x,y][x^n,y,x][x^n,y,x,[x,y]]\\&=[x,y]^{n+1}[x,y,x]^{\binom{n}{2}+n}[x,y,x,x]^{\binom{n}{2}+\binom{n}{3}}\\&\qquad [x,y,x,x,x]^{\binom{n}{3}+\binom{n}{4}}[x,y,x,[x,y]]^{a(n+1)}[x,y,x,x,x,x]^{\binom{n}{4}+\binom{n}{5}}\\ &\qquad  [x,y,x,[x,y],x]^{\binom{n}{2}+3{\binom{n}{3}}+2{\binom{n}{4}}}[x,y,x,x,[x,y]]^{\binom{n}{2}+2\binom{n}{3}+\binom{n}{4}}.
\end{align*}
Hence our conclusion follows.
\end{proof}

\begin{definition}\cite[Definition 5]{25''}\label{p2.5}
$G$ is polycyclic if it has a descending chain of subgroups
$G = G_1 \geq G_2 \geq ... \geq G_{n+1} = 1$ in which $G_{i+1} \lhd G_i$, and ${G_i}/{G_{i+1}}$ is cyclic.
Such a chain of subgroups is called a polycyclic series. 
\end{definition}
Since ${G_i}/{G_{i+1}}$ is cyclic, there exist $x_i$ with $<{x_i}{G_{i+1}}>={G_i}/{G_{i+1}}$ for every $i\in \{1,...,n\}$.

\begin{definition}\cite[Definition 7]{25''}
A polycyclic generating sequence of a group $G$ (PCGS) is a sequence $x_1,...,x_n$ of elements of $G$ such that $<{x_i}G_{i+1}>={G_i}/{G_{i+1}}$ for $1\leq i \leq n$. 
\end{definition}

\begin{definition}\cite[Definition 8]{25''}
Let $X$ be a PCGS for $G$. The sequence $R(X):=
(r_1,...,r_n)$ defined by $r_i:=|G_i:G_{i+1}|\in \Bbb{N}\cup \{\infty\}$ is the sequence of relative orders for $X$.
\end{definition}
Note that $G$ is finite if and only if every entry in $R(X)$ is finite.

\begin{lemma}\cite[Lemma 12]{25''}
Let $X = {x_1,...,x_n}$ be a polycyclic sequence for $G$ with the
relative orders $R(X) = (r_1,...,r_n)$. For every $g\in G$ there exists a sequence $(e_1,...,e_n)$, with $e_i\in \Bbb{Z}$ for $1 \leq i \leq n$ and $0 \leq e_i < r_i$, such that $g=x_{1}^{e_1},...,x_{n}^{e_n}$.
\end{lemma}

\begin{definition}\cite[Definition 14]{25''}\label{p2.55}
The expression $g={x_1}^{e_1}\ldots {x_n}^{e_n}$ is the normal form of $g\in G$ with respect to $X$.
\end{definition}

\begin{definition}\cite[Definition 16]{25''}\label{p2.56}
A presentation $\{x_1,...,x_n | R\}$ is a polycyclic presentation, if there is a sequence $S=(s_1,...,s_n)$ with $s_i\in \Bbb{N}\cup \{\infty\}$ and integers $a_{i,k} , b_{i,j,k} , c_{i,j,k}$ such that $R$ consists of the following relations
$${x_i}^{s_i}={x_{i+1}}^{a_{i,i+1}}...{x_n}^{a_{i,n}} \ \ \ \text{for} \ \ \  1\leq i\leq n \ \ \text{with}\ \  {s_i}<\infty,$$
$${{x_j}^{-1}}{x_i}{x_j}={x_{i+1}}^{b_{i,j,j+1}}...{x_n}^{b_{i,j,n}} \ \ \  \text{for} \ \ \  1\leq j<i\leq n,$$
$${x_j}{x_i}{{x_j}^{-1}}={x_{j+1}}^{c_{i,j,j+1}}...{x_n}^{c_{i,j,n}} \ \ \ \text{for} \ \ \ 1\leq j<i\leq n.$$
\end{definition}
If $G$ is  defined by such a polycyclic presentation, then $G$ is called a $PC$ group. In addition to every $PC$ group can be deﬁned by a polycyclic presentation.

\begin{definition}\cite{25''}\label{p2.57}
A polycyclic presentation in which every element is represented by exactly one normal word is consistent. 
\end{definition}

In this paper, we will also use of consistent presentation in Section $4$ to calculate the Bogomolov multiplier of some $p$-groups of order $p^7$ and exponent $p$.

\begin{proposition}\cite[Proposition 20]{1}\label{p2.6}
Let $G$ be a finite group with a polycyclic generating sequence $x_1,\ldots ,x_n$, then the group $[G,G^{\varphi}]$ is generated by $$\{[x_i,{x_j}^{\varphi}] \ | \  i,j=1,\ldots ,n , i>j\}.$$
\end{proposition}
\begin{proposition}\cite[Proposition 3.2]{23}\label{p2.7}
Let $p>3$ be a prime number, $G$ be a $p$-group of class at most $3$ and $x_1,\ldots ,x_n$ be a polycyclic generating sequence of $G$. If all nontrivial commutators $[x_i,x_j]$ $(i>j)$ are different elements of the polycyclic generating sequence, then ${B_0}(G)=0$.
\end{proposition}

\section{\bf{Groups of order $p^7$ and exponent $p$ $(p>5)$ with trivial Bogomolov multiplier}}
The groups of exponent $p$ and order $p^7$, were classified by Wilkinson in \cite{31}. Each group is presented by seven generators, named alphabetically from $a$ to $g$. The non trivial commutator relations between generators are introduced as $[b, a] = cg$, the right side of this relation is the usual product of the elements of the group. It is assumed that all other commutators are identity, and that the $p$th power of every generator is identity. Using the notations of \cite{31}, we have the following propositions.
\\
Here, we use various techniques to determine group with trivial Bogomolov multiplier depending on the structure of the group, so we state the results in separate propositions. 
\begin{proposition}\label{p3.1}
 The following groups, have trivial Bogomolov multiplier.
\begin{enumerate}[]
\item{$G_1: \text{Elementary abelian groups}$}.
\item{$G_2: <a,b,c,d,e,f,g \ | \  [b,a]=c>$}.
\item{$G_3: <a,b,c,d,e,f,g \ | \  [b,a]=c , [c,a]=d>$}.
\item{$G_4: <a,b,c,d,e,f,g \ | \  [b,a]=d , [c,a]=e>$}.
\item{$G_5: <a,b,c,d,e,f,g \ | \  [b,a]=c , [c,a]=d , [c,b]=e>$}.
\item{$G_8: <a,b,c,d,e,f,g \ | \  [b,a]=c , [c,a]=d , [d,a]=e>$}.
\item{$G_{10}: <a,b,c,d,e,f,g \ | \  [b,a]=d , [c,a]=e , [c,b]=f>$}.
\item{$G_{11}: <a,b,c,d,e,f,g \ | \  [b,a]=e , [d,c]=f>$}.
\item{$G_{15}: <a,b,c,d,e,f,g \ | \  [b,a]=c , [c,a]=e , [d,b]=f>$}.
\item{$G_{17}: <a,b,c,d,e,f,g \ | \  [b,a]=c , [c,a]=e , [d,a]=f>$}.
\item{$G_{35}: <a,b,c,d,e,f,g \ | \  [b,a]=e , [c,a]=f , [d,c]=g>$}.
\item{$G_{39}: <a,b,c,d,e,f,g \ | \  [b,a]=e , [c,a]=f , [d,a]=g>$}.
\item{$G_{40}: <a,b,c,d,e,f,g \ | \  [b,a]=c , [c,a]=e , [d,a]=f , [d,b]=g>$}.
\item{$G_{41}: <a,b,c,d,e,f,g \ | \  [b,a]=c , [c,a]=e , [c,b]=f , [d,a]=g>$}.
\item{$G_{45}: <a,b,c,d,e,f,g \ | \  [b,a]=c , [c,a]=f , [e,d]=g>$}.
\end{enumerate}
\end{proposition}
\begin{proof}
These groups are nilpotent of class at most $3$ and all nontrivial commutators $[x_i,x_j]$ $(i>j)$ are different elements of the polycyclic generating sequence. Thus, by using Proposition \ref{p2.7}, they have trivial Bogomolov multiplier.
\end{proof}
\begin{proposition}\label{p3.2}
The following groups, have trivial Bogomolov multiplier.
\begin{enumerate}[]
\item{$G_{20}: <a,b,c,d,e,f,g \ | \  [b,a]=c , [c,a]=d , [c,b]=e , [d,a]=f>$}.
\item{$G_{30}: <a,b,c,d,e,f,g \ | \  [b,a]=c , [c,a]=d , [d,a]=e , [e,a]=f>$}.
\item{$G_{60}: <a,b,c,d,e,f,g \ | \  [b,a]=c , [c,a]=d , [d,a]=f , [e,b]=g>$}.
\item{$G_{64}: <a,b,c,d,e,f,g \ | \  [b,a]=c , [c,a]=d , [d,a]=f , [e,a]=g>$}.
\item{$G_{70}: <a,b,c,d,e,f,g \ | \  [b,a]=d , [c,a]=e , [d,b]=f , [e,c]=g>$}.
\item{$G_{75}: <a,b,c,d,e,f,g \ | \  [b,a]=d , [c,a]=e , [d,a]=f , [e,c]=g>$}.
\item{$G_{85}: <a,b,c,d,e,f,g \ | \  [b,a]=d , [c,a]=e , [d,a]=f , [e,a]=g>$}.
\item{$G_{95}: <a,b,c,d,e,f,g \ | \  [b,a]=c , [c,a]=d , [c,b]=e , [d,a]=f , [e,b]=g>$}.
\item{$G_{103}: <a,b,c,d,e,f,g \ | \  [b,a]=c , [c,a]=d , [c,b]=f , [d,a]=e , [e,a]=g>$}.
\item{$G_{190}: <a,b,c,d,e,f,g \ | \  [b,a]=c , [c,a]=d , [d,a]=e , [e,a]=f , [f,a]=g>$}.
\end{enumerate}
\end{proposition}
\begin{proof}
We state the proof in details to $G_{20}$ and $G_{190}$, the remaining case can be proved in similar way. 
\\
Let $G\cong {G_{20}}$. By using Proposition \ref{p2.6}, $[G_{20},{G_{20}^{\varphi}}]$ is generated by $[b,a^{\varphi}] ,[c,a^{\varphi}]$, $[c,b^{\varphi}] , [d,a^{\varphi}] $ modulo ${\mathcal{M}_0}^{*}(G_{20})$. Using Lemma \ref{l2.3} (vi), we have $$[[b,a^{\varphi}],[c,a^{\varphi}]]=[[b,a],[c,a]^{\varphi}]=[c,d^{\varphi}]\in {\mathcal{M}_0}^{*}(G_{20}),$$
 and
$$[[b,a^{\varphi}],[c,b^{\varphi}]]=[[b,a],[c,b]^{\varphi}]=[c,e^{\varphi}]\in {\mathcal{M}_0}^{*}(G_{20}).$$
Similarly,
$$[[b,a^{\varphi}],[d,a^{\varphi}]], [[c,a^{\varphi}],[c,b^{\varphi}]], [[c,a^{\varphi}],[d,a^{\varphi}]], [[c,b^{\varphi}],[d,a^{\varphi}]]\in {\mathcal{M}_0}^{*}(G_{20}).$$
Thus any two elements of the generating set of $[G_{20},G_{20}^{\varphi}]$, are commuting modulo ${\mathcal{M}_0}^{*}(G_{20})$, and each element of $[G_{20},G_{20}^{\varphi}]$ can be expressed as
$$[b,a^{\varphi}]^{{\alpha}_1} [c,a^{\varphi}]^{{\alpha}_2} [c,b^{\varphi}]^{{\alpha}_3} [d,a^{\varphi}]^{{\alpha}_4} {\tilde{w}},$$
where ${\tilde{w}}\in {\mathcal{M}_0}^{*}(G_{20})$, and $1\leq i\leq 4$, ${{\alpha}_i}\in {\Bbb{Z}}$. Let $w=[b,a^{\varphi}]^{{\alpha}_1} [c,a^{\varphi}]^{{\alpha}_2} [c,b^{\varphi}]^{{\alpha}_3} [d,a^{\varphi}]^{{\alpha}_4} {\tilde{w}} \in \mathcal{M}^{*}(G_{20})$, then $1={\kappa}^{*}(w)=c^{{\alpha}_1}d^{{\alpha}_2}e^{{\alpha}_3}f^{{\alpha}_4}$. Since $c,d,e,f$ are in the polycyclic generating sequence and $exp(G_{20})=p$, we have $c^{{\alpha}_1}=d^{{\alpha}_2}=e^{{\alpha}_3}=f^{{\alpha}_4}=1$ and $p$ divides ${\alpha}_1,{\alpha}_2,{\alpha}_3$ and ${\alpha}_4$, respectively. Now using Lemmas \ref{l2.3} and  \ref{l2.4}, we have
$$1=[c^p,b^{\varphi}]=[c,b^{\varphi}]^{p}[c,b^{\varphi},c]^{{\binom{p}{2}}}[c,b^{\varphi},c,c]^{{\binom{p}{3}}}.$$ 
Since $$[c,b^{\varphi},c]^{{\binom{p}{2}}}=[c,b,c^{\varphi}]^{{\binom{p}{2}}}=[e,c^{\varphi}]^{{\binom{p}{2}}}=[e^{{\binom{p}{2}}},c^{\varphi}]=1$$
and
$$[c,b^{\varphi},c,c]^{{\binom{p}{3}}}=[c,b,c,c^{\varphi}]^{{\binom{p}{3}}}=[e,c,c^{\varphi}]^{{\binom{p}{3}}}=[1,c^{\varphi}]^{{\binom{p}{3}}}=1,$$ $ [c,b^{\varphi}]^{p}=1$. Similarly, $[b,a^{\varphi}]^{p}=[c,a^{\varphi}]^{p}=[d,a^{\varphi}]^{p}=1$. Thus $w=\tilde{w}$. Hence $\mathcal{M}^{*}(G_{20}) \subseteq  {\mathcal{M}_0}^{*}(G_{20})$ and $\tilde{B_0}(G_{20})=0$.
\\
\\
Now let $G\cong G_{190}$. By using Proposition \ref{p2.6}, the group $[G_{190},G_{190}^{\varphi}]$ is generated by $[b,a^{\varphi}] ,[c,a^{\varphi}] , [d,a^{\varphi}] , [e,a^{\varphi}],$ and $[f,a^{\varphi}] $ modulo ${\mathcal{M}_0}^{*}(G_{190})$. Lemma \ref{l2.3} (vi) implies $$[[b,a^{\varphi}],[c,a^{\varphi}]]=[c,d^{\varphi}]\in {\mathcal{M}_0}^{*}(G_{190}).$$ Also
$$[[b,a^{\varphi}],[d,a^{\varphi}]]=[c,e^{\varphi}]\in {\mathcal{M}_0}^{*}(G_{190}).$$
Similarly,
\begin{align*}[[b,a^{\varphi}],[e,a^{\varphi}]]&,[[b,a^{\varphi}],[f,a^{\varphi}]],[[c,a^{\varphi}],[d,a^{\varphi}]],[[c,a^{\varphi}],[e,a^{\varphi}]],[[c,a^{\varphi}],[f,a^{\varphi}]],\\
&[[d,a^{\varphi}],[e,a^{\varphi}]],[[d,a^{\varphi}],[f,a^{\varphi}]],[[e,a^{\varphi}],[f,a^{\varphi}]]\in {\mathcal{M}_0}^{*}(G_{190}).
\end{align*}
Thus any two elements of generating set of $[G_{190},G_{190}^{\varphi}]$ are commuting modulo ${\mathcal{M}_0}^{*}(G_{190})$. So each element of $[G_{190},G_{190}^{\varphi}]$ can be written as
$$[b,a^{\varphi}]^{{\alpha}_1} [c,a^{\varphi}]^{{\alpha}_2} [d,a^{\varphi}]^{{\alpha}_3} [e,a^{\varphi}]^{{\alpha}_4}[f,a^{\varphi}]^{{\alpha}_5}\tilde{w},$$
where $\tilde{w}\in {\mathcal{M}_0}^{*}(G_{190})$, and $1\leq i\leq 5$, ${{\alpha}_i}\in {\Bbb{Z}}$. \\ Let $w={[b,a^{\varphi}]^{{\alpha}_1} [c,a^{\varphi}]^{{\alpha}_2} [d,a^{\varphi}]^{{\alpha}_3} [e,a^{\varphi}]^{{\alpha}_4}[f,a^{\varphi}]^{{\alpha}_5}}\tilde{w} \in \mathcal{M}^{*}(G_{190})$. Then \\ $1={\kappa}^{*}(w)=c^{{\alpha}_1}d^{{\alpha}_2}e^{{\alpha}_3}f^{{\alpha}_4}g^{{\alpha}_5}$. Since $c,d,e,f,g$ belong to the polycyclic generating sequence and $exp(G_{190})=p$, we obtain $c^{{\alpha}_1}=d^{{\alpha}_2}=e^{{\alpha}_3}=f^{{\alpha}_4}=g^{{\alpha}_5}=1$ and $p$ divides ${\alpha}_1,{\alpha}_2,{\alpha}_3,{\alpha}_4$ and ${\alpha}_5$, respectively. Now  using Lemmas \ref{l2.3} and  \ref{l2.4}, we have
\begin{align*}1=[b^p,a^{\varphi}]=&[b,a^{\varphi}]^p[b,a^{\varphi},b]^{\binom{p}{2}}[b,a^{\varphi},b,b]^{\binom{p}{3}}[b,a^{\varphi},b,b,b]^{\binom{p}{4}}[b,a^{\varphi},b,[b,a^{\varphi}]]^{\binom{p}{3}}\\&[b,a^{\varphi},b,b,b,b]^{\binom{p}{5}}[b,a^{\varphi},b,b,[b,a^{\varphi}]]^{\binom{p}{4}+\binom{p}{3}}[b,a^{\varphi},b,[b,a^{\varphi}],b]^{\binom{p}{4}}.\end{align*}
Since $$[b,a^{\varphi},b]^{{\binom{p}{2}}}=[b,a,b^{\varphi}]^{{\binom{p}{2}}}=[c,b^{\varphi}]^{{\binom{p}{2}}}=[c^{{\binom{p}{2}}},b^{\varphi}]=1,$$  $$[b,a^{\varphi},b,b]^{{\binom{p}{3}}}=[b,a,b,b^{\varphi}]^{{\binom{p}{3}}}=[c,b,b^{\varphi}]^{{\binom{p}{3}}}=[1,b^{\varphi}]^{{\binom{p}{3}}}=1,$$ 
and
\begin{align*}[b,a^{\varphi},b,b,b]^{\binom{p}{4}}&=[b,a^{\varphi},b,[b,a^{\varphi}]]^{\binom{p}{3}}=[b,a^{\varphi},b,b,b,b]^{\binom{p}{5}}\\&=[b,a^{\varphi},b,b,[b,a^{\varphi}]]^{\binom{p}{4}+\binom{p}{3}}=[b,a^{\varphi},b,[b,a^{\varphi}],b]^{\binom{p}{4}}=1,
\end{align*}
we have $ [b,a^{\varphi}]^{p}=1$. Similarly $[c,a^{\varphi}]^{p}=[d,a^{\varphi}]^{p}=[e,a^{\varphi}]^{p}=[f,a^{\varphi}]^p=1$. \\ So, $\mathcal{M}^{*}(G_{190})\subseteq {\mathcal{M}_0}^{*}(G_{190})$. Hence $\tilde{B_0}(G_{190})=0$.

\end{proof}
\begin{proposition}\label{p3.3}
The following groups, have trivial Bogomolov multiplier.
\begin{enumerate}[]
\item{$G_{6}: <a,b,c,d,e,f,g \ | \  [b,a]=e , [d,c]=e>$}.
\item{$G_{7}: <a,b,c,d,e,f,g \ | \  [b,a]=c , [c,a]=[d,b]=e>$}.
\item{$G_{12}: <a,b,c,d,e,f,g \ | \  [b,a]=e , [c,a]=f , [d,c]=e>$}.
\item{$G_{16}: <a,b,c,d,e,f,g \ | \  [b,a]=c , [c,a]=e , [c,b]=f , [d,b]=f>$}.
\item{$G_{21}: <a,b,c,d,e,f,g \ | \  [b,a]=c , [c,a]=f , [e,d]=f>$}.
\item{$G_{22}: <a,b,c,d,e,f,g \ | \  [b,a]=c , [c,a]=d , [d,a]=f , [e,b]=f>$}.
\item{$G_{24}: <a,b,c,d,e,f,g \ | \  [b,a]=d , [c,a]=e , [d,b]=f , [e,c]=f>$}.
\item{$G_{25}: <a,b,c,d,e,f,g \ | \  [b,a]=d , [c,a]=e , [d,b]=f , [e,c]=f^t>$}.
\item{$G_{26}: <a,b,c,d,e,f,g \ | \  [b,a]=d , [c,a]=e , [d,b]=f , [e,a]=f>$}.
\item{$G_{28}: <a,b,c,d,e,f,g \ | \  [b,a]=c , [c,a]=d , [c,b]=e , [d,a]=f , [e,b]=f>$}.
\item{$G_{32}: <a,b,c,d,e,f,g \ | \  [b,a]=c , [c,a]=d , [d,a]=e , [d,b]=f , [e,b]=f , [d,c]=f^{-1}>$}.
\item{$G_{36}: <a,b,c,d,e,f,g \ | \  [b,a]=e , [c,a]=f , [d,b]=f , [d,c]=g>$}.
\item{$G_{37}: <a,b,c,d,e,f,g \ | \  [b,a]=e , [c,a]=f , [c,b]=g , [d,b]=f>$}.
\item{$G_{43}: <a,b,c,d,e,f,g \ | \  [b,a]=f , [d,c]=g , [e,d]=f>$}.
\item{$G_{46}: <a,b,c,d,e,f,g \ | \  [b,a]=c , [c,a]=f , [d,b]=f , [e,d]=g>$}.
\item{$G_{47}: <a,b,c,d,e,f,g \ | \  [b,a]=c , [c,a]=f , [d,b]=g , [e,d]=f>$}.
\item{$G_{48}: <a,b,c,d,e,f,g \ | \  [b,a]=c , [c,a]=f , [d,b]=g , [e,b]=f>$}.
\item{$G_{49}: <a,b,c,d,e,f,g \ | \  [b,a]=c , [c,a]=f , [d,a]=g , [e,d]=f>$}.
\item{$G_{50}: <a,b,c,d,e,f,g \ | \  [b,a]=c , [c,a]=f , [d,a]=g , [e,b]=f>$}.
\item{$G_{52}: <a,b,c,d,e,f,g \ | \  [b,a]=c , [c,a]=f , [d,a]=g , [e,b]=g>$}.
\item{$G_{54}: <a,b,c,d,e,f,g \ | \  [b,a]=c , [c,a]=f , [c,b]=g , [e,d]=f>$}.
\item{$G_{55}: <a,b,c,d,e,f,g \ | \  [b,a]=c , [c,a]=f , [c,b]=g , [d,b]=f , [e,d]=g>$}.
\item{$G_{57}: <a,b,c,d,e,f,g \ | \  [b,a]=c , [c,a]=f , [c,b]=g , [d,a]=g , [d,b]=f , [e,b]=g>$}.
\item{$G_{59}: <a,b,c,d,e,f,g \ | \  [b,a]=c , [c,a]=d , [c,b]=g , [d,a]=f , [e,b]=f>$}.
\item{$G_{62}: <a,b,c,d,e,f,g \ | \  [b,a]=c , [c,a]=d , [c,b]=g , [d,a]=f , [e,b]=g>$}.
\item{$G_{71}: <a,b,c,d,e,f,g \ | \  [b,a]=d , [c,a]=e , [c,b]=g , [d,b]=f , [e,c]=f>$}.
\item{$G_{73}: <a,b,c,d,e,f,g \ | \  [b,a]=d , [c,a]=e , [d,a]=f , [d,b]=g , [e,c]=f>$}.
\item{$G_{74}: <a,b,c,d,e,f,g \ | \  [b,a]=d , [c,a]=e , [c,b]=g , [d,a]=f , [e,c]=f>$}.
\item{$G_{76}: <a,b,c,d,e,f,g \ | \  [b,a]=d , [c,a]=e , [d,a]=f , [d,b]=g , [e,c]=g>$}.
\item{$G_{81}: <a,b,c,d,e,f,g \ | \  [b,a]=d , [c,a]=e , [d,a]=f , [e,a]=f , [e,c]=g , d,b]=g^{-1}>$}.
\item{$G_{89}: <a,b,c,d,e,f,g \ | \  [b,a]=d , [c,a]=e , [d,c]=f , [e,b]=f , [e,c]=g>$}.
\item{$G_{91}: <a,b,c,d,e,f,g \ | \  [b,a]=d , [c,a]=e , [d,b]=g , [d,c]=f , [e,a]=f , [e,b]=f>$}.
\item{$G_{100}: <a,b,c,d,e,f,g \ | \  [b,a]=c , [c,a]=e , [c,b]=g , [d,c]=f , [e,a]=f , [d,b]=eg>$}.
\item{$G_{105}: <a,b,c,d,e,f,g \ | \  [b,a]=c , [c,a]=d , [d,a]=e , [d,b]=f , [e,a]=g , [e,b]=f , [d,c]=f^{-1}>$}.
\item{$G_{110}: <a,b,c,d,e,f,g \ | \  [b,a]=g , [d,c]=g , [f,e]=g >$}.
\item{$G_{111}: <a,b,c,d,e,f,g \ | \  [b,a]=c , [c,a]=g , [d,b]=g , [f,e]=g>$}.
\item{$G_{112}: <a,b,c,d,e,f,g \ | \  [b,a]=c , [c,a]=d , [d,a]=g , [f,e]=g>$}.
\item{$G_{114}: <a,b,c,d,e,f,g \ | \  [b,a]=d , [c,a]=e , [d,a]=g , [e,c]=g , [f,b]=g>$}.
\item{$G_{117}: <a,b,c,d,e,f,g \ | \  [b,a]=d , [c,a]=e , [d,c]=g , [e,b]=g , [f,a]=g>$}.
\item{$G_{118}: <a,b,c,d,e,f,g \ | \  [b,a]=c , [c,a]=e , [d,b]=e , [d,c]=g , [e,a]=g , [f,d]=g>$}.
\item{$G_{120}: <a,b,c,d,e,f,g \ | \  [b,a]=c , [c,a]=e , [d,b]=e , [d,c]=g , [e,a]=g , [f,b]=g , [f,d]=g>$}.
\item{$G_{121}: <a,b,c,d,e,f,g \ | \  [b,a]=c , [c,a]=d , [d,a]=e , [e,a]=g , [f,b]=g>$}.
\item{$G_{122}: <a,b,c,d,e,f,g \ | \  [b,a]=c , [c,a]=d , [c,b]=g , [d,a]=e , [e,b]=g , [f,b]=g>$}.
\item{$G_{123}: <a,b,c,d,e,f,g \ | \  [b,a]=c , [c,a]=d , [d,a]=e , [d,b]=g , [e,b]=g , [f,a]=g , [d,c]=g^{-1}>$}.
\item{$G_{131}: <a,b,c,d,e,f,g \ | \  [b,a]=e , [d,c]=f , [e,a]=g , [f,c]=g>$}.
\item{$G_{132}: <a,b,c,d,e,f,g \ | \  [b,a]=e , [d,b]=g , [d,c]=f , [e,a]=g , [f,c]=g>$}.
\item{$G_{140}: <a,b,c,d,e,f,g \ | \  [b,a]=c , [c,a]=e , [d,b]=f , [e,a]=g , [f,d]=g>$}.

\item{$G_{141}: <a,b,c,d,e,f,g \ | \  [b,a]=c , [c,a]=e , [c,b]=g , [d,b]=f , [e,a]=g>$}.
\item{$G_{142}: <a,b,c,d,e,f,g \ | \  [b,a]=c , [c,a]=e , [d,b]=f , [e,a]=g , [f,b]=g>$}.
\item{$G_{143}: <a,b,c,d,e,f,g \ | \  [b,a]=c , [c,a]=e , [d,b]=f , [d,c]=g , [e,a]=g , [f,a]=g , [f,b]=g>$}.
\item{$G_{144}: <a,b,c,d,e,f,g \ | \  [b,a]=c , [c,a]=e , [c,b]=f , [d,b]=f , [e,a]=g , [f,b]=g>$}.
\item{$G_{149}: <a,b,c,d,e,f,g \ | \  [b,a]=c , [c,a]=e , [d,a]=f , [e,a]=g , [f,d]=g>$}.
\item{$G_{152}: <a,b,c,d,e,f,g \ | \  [b,a]=c , [c,a]=e , [d,a]=f , [e,a]=g , [f,b]=g , [d,c]=g^{-1}>$}.
\item{$G_{162}: <a,b,c,d,e,f,g \ | \  [b,a]=c , [c,a]=d , [c,b]=e , [d,a]=f , [e,b]=g , [f,a]=g>$}.
\item{$G_{163}: <a,b,c,d,e,f,g \ | \  [b,a]=c , [c,a]=d , [c,b]=e , [d,a]=f , [f,b]=g , [d,c]=g^{-1} , [e,a]=g^{-1}>$}.
\item{$G_{164}: <a,b,c,d,e,f,g \ | \  [b,a]=c , [c,a]=d , [c,b]=e , [d,a]=f , [e,b]=g , [f,b]=g , [d,c]=g^{-1} , [e,a]=g^{-1}>$}.
\item{$G_{168}: <a,b,c,d,e,f,g \ | \  [b,a]=c , [c,a]=d , [d,a]=f , [d,b]=g , [e,a]=g , [e,b]=f , [f,b]=g , [d,c]=g^{-1}>$}.
\end{enumerate}
\end{proposition}
\begin{proof}
We state the proof in details to $G_{12}$, $G_{16}$ and $G_{132}$, the remaining case can be proved in similar way. 
\\
Let $G\cong G_{12}$. We can see that $[G_{12},G_{12}^{\varphi}]$ is generated by $[b,a^{\varphi}] ,[c,a^{\varphi}] , [d,c^{\varphi}]$ modulo ${\mathcal{M}_0}^{*}(G_{12})$. Using Lemma \ref{l2.3}  (vi), we have $$[[b,a^{\varphi}],[c,a^{\varphi}]]=[[b,a],[c,a]^{\varphi}]=[e,f^{\varphi}]\in {\mathcal{M}_0}^{*}(G_{12})$$
 and
$$[[b,a^{\varphi}],[d,c^{\varphi}]]=[[b,a],[d,c]^{\varphi}]=[e,e^{\varphi}]\in {\mathcal{M}_0}^{*}(G_{12}).$$
Also
$$[[c,a^{\varphi}],[d,c^{\varphi}]]=[[c,a],[d,c]^{\varphi}]=[f,e^{\varphi}]\in {\mathcal{M}_0}^{*}(G_{12}).$$
Hence any two elements of the generator set of $[G_{12},G_{12}^{\varphi}]$ are commuting modulo ${\mathcal{M}_0}^{*}(G_{12})$. By using Proposition \ref{p2.6}, each element of $[G_{12},G_{12}^{\varphi}]$ can be written as 
$$[b,a^{\varphi}]^{{\alpha}_1} [c,a^{\varphi}]^{{\alpha}_2} [d,c^{\varphi}]^{{\alpha}_3} {\tilde{w}},$$ where ${\tilde{w}}\in {\mathcal{M}_0}^{*}(G_{12})$, and for all $i$ such that, $1\leq i\leq 3$, ${{\alpha}_i}\in {\Bbb{Z}}$. Let $w=[b,a^{\varphi}]^{{\alpha}_1} [c,a^{\varphi}]^{{\alpha}_2} [d,c^{\varphi}]^{{\alpha}_3}{\tilde{w}} \in \mathcal{M}^{*}(G_{12})$. Then $1={\kappa}^{*}(w)=e^{{\alpha}_1}f^{{\alpha}_2}e^{{\alpha}_3}=e^{{\alpha}_{1}+{\alpha}_{3}}f^{{\alpha}_2}$. Since $e$ and $f$ are  elements of polycyclic generating sequence and $exp(G_{12})=p$, we have $f^{{\alpha}_2}=e^{{{\alpha}_1}+{\alpha}_3}=1$ and $p$ divides ${\alpha}_2$ and $({\alpha}_1+{\alpha}_3)$, respectively. So, there is two integers, $k$ and $k'$ such that ${\alpha}_2=kp$ and ${\alpha}_1+{\alpha}_3=k'p$. Thus 
\begin{align*}w&=[b,a^{\varphi}]^{{{\alpha}_1}} [c,a^{\varphi}]^{{\alpha}_2}[d,c^{\varphi}]^{k'p-{\alpha}_1}{\tilde{w}}\\&={([b,a^{\varphi}][d,c^{\varphi}]^{-1})}^{{\alpha}_1}[d,c^{\varphi}]^{k'p} [c,a^{\varphi}]^{kp}{\tilde{w}}\\
&={([b,a^{\varphi}][d,c^{\varphi}]^{-1})}^{{\alpha}_1}{\tilde{w}}.
\end{align*}
We know that $cl(G_{12})=2$. So, Lemmas \ref{l2.3} and \ref{l2.4} imply that
$$1=[c^p,a^{\varphi}]=[c,a^{\varphi}]^{p} \ \ \ \ , \ \ \ \ 1=[d^p,c^{\varphi}]=[d,c^{\varphi}]^{p}.$$ 
Thus $w=([b,a^{\varphi}][d,c^{\varphi}]^{-1})^{{\alpha}_1}{\tilde{w}}$. We claim that $[b,a^{\varphi}][d,c^{\varphi}]^{-1}\in {\mathcal{M}_0}^{*}(G_{12})$, and so the result follows. At first we use Lemma \ref{l2.3} (i) to check that $[cab,cad]=1$. Indeed,
\begin{align*}[cab,cad]&=[ca,cad][ca,cad,b][b,cad]\\&=[ca,d][ca,ca][[ca,ca],d][b,d][b,ca][b,ca,d]\\&=[c,d][c,d,a][a,d][b,d][b,a][b,c][b,c,a][[b,a][b,c][b,c,a],d]=e^{-1}e=1.
\end{align*}
Thus $[cab,(cad)^{\varphi}]=[cab,c^{\varphi}a^{\varphi}d^{\varphi}]\in {\mathcal{M}_0}^{*}(G_{12})$. Expanding it, we obtain that 
\begin{align*}[cab,c^{\varphi}a^{\varphi}d^{\varphi}]&=[c,d^{\varphi}][c,d^{\varphi},a][a,d^{\varphi}][b,d^{\varphi}][b,a^{\varphi}][b,c^{\varphi}]\\&[b,c^{\varphi},a][[b,a^{\varphi}][b,c^{\varphi}][b,c^{\varphi},a],d^{\varphi}].
\end{align*}
We can see that, all of the above commutators except $[c,d^{\varphi}]$ and $[b,a^{\varphi}]$, belong to ${\mathcal{M}_0}^{*}(G_{12})$. So modulo ${\mathcal{M}_0}^{*}(G_{12})$,
$[c,d^{\varphi}][b,a^{\varphi}]=[b,a^{\varphi}][c,d^{\varphi}]=[b,a^{\varphi}][d,c^{\varphi}]^{-1}$, and $[b,a^{\varphi}][d,c^{\varphi}]^{-1}\in {\mathcal{M}_0}^{*}(G_{12})$, as required. Hence $\tilde{B_0}(G_{12})=0$. 
\\
\\
Let $G\cong G_{16}$. We can see that $[G_{16},G_{16}^{\varphi}]$ is generated by $[b,a^{\varphi}] ,[c,a^{\varphi}] , [d,b^{\varphi}] , [c,b^{\varphi}] $ modulo ${\mathcal{M}_0}^{*}(G_{16})$. Using Lemma \ref{l2.3}  (vi), we have $$[[b,a^{\varphi}],[c,a^{\varphi}]]=[[b,a],[c,a]^{\varphi}]=[c,e^{\varphi}]\in {\mathcal{M}_0}^{*}(G_{16})$$
 and
$$[[b,a^{\varphi}],[d,b^{\varphi}]]=[[b,a],[d,b]^{\varphi}]=[c,f^{\varphi}]\in {\mathcal{M}_0}^{*}(G_{16}).$$
Similarly,
$$[[b,a^{\varphi}],[c,b^{\varphi}]], [[c,a^{\varphi}],[d,b^{\varphi}]], [[c,a^{\varphi}],[c,b^{\varphi}]], [[d,b^{\varphi}],[c,b^{\varphi}]]\in {\mathcal{M}_0}^{*}(G_{16}).$$
Hence any two elements of the generator set of $[G_{16},G_{16}^{\varphi}]$ are commuting modulo ${\mathcal{M}_0}^{*}(G_{16})$. By using Proposition \ref{p2.6}, each element of $[G_{16},G_{16}^{\varphi}]$ can be written as 
$$[b,a^{\varphi}]^{{\alpha}_1} [c,a^{\varphi}]^{{\alpha}_2} [d,b^{\varphi}]^{{\alpha}_3} [c,b^{\varphi}]^{{\alpha}_4} {\tilde{w}},$$ where ${\tilde{w}}\in {\mathcal{M}_0}^{*}(G_{16})$, and for all $i$ such that, $1\leq i\leq 4$, ${{\alpha}_i}\in {\Bbb{Z}}$. Let $w=[b,a^{\varphi}]^{{\alpha}_1} [c,a^{\varphi}]^{{\alpha}_2} [d,b^{\varphi}]^{{\alpha}_3} [c,b^{\varphi}]^{{\alpha}_4} {\tilde{w}} \in \mathcal{M}^{*}(G_{16})$. Then $1={\kappa}^{*}(w)=c^{{\alpha}_1}e^{{\alpha}_2}f^{{\alpha}_3}f^{{\alpha}_4}=c^{{\alpha}_1}e^{{\alpha}_2}f^{{\alpha}_{3}+{\alpha}_{4}}$. Since $c$, $e$ and $f$ are  elements of polycyclic generating sequence and $exp(G_{16})=p$, we have $c^{{\alpha}_1}=e^{{\alpha}_2}=f^{{{\alpha}_3}+{\alpha}_4}=1$ and hence $p$ divides ${\alpha}_1,{\alpha}_2,$ and $({\alpha}_3+{\alpha}_4)$, respectively.
We know $cl(G_{16})=3$, so by using Lemmas \ref{l2.3} and \ref{l2.4}, we have
$$1=[c^p,b^{\varphi}]=[c,b^{\varphi}]^{p}[c,b^{\varphi},c]^{{\binom{p}{2}}}.$$ 
Since $[c,b^{\varphi},c]^{{\binom{p}{2}}}=[c,b,c^{\varphi}]^{{\binom{p}{2}}}=[f,c^{\varphi}]^{{\binom{p}{2}}}=[f^{{\binom{p}{2}}},c^{\varphi}]=1$, $[c,b^{\varphi}]^p=1$. Also
$$1=[b^p,a^{\varphi}]=[b,a^{\varphi}]^{p}[b,a^{\varphi},b]^{{\binom{p}{2}}},$$ and
$$[b,a^{\varphi},b]^{{\binom{p}{2}}}=[b,a,b^{\varphi}]^{{\binom{p}{2}}}=[c,b^{\varphi}]^{{\binom{p}{2}}}=[c^{{{\binom{p}{2}}}},b^{\varphi}]=1.$$
Thus $[b,a^{\varphi}]^{p}=1$. Similarly $[c,a^{\varphi}]^{p}=1$, and so $w=([d,b^{\varphi}][c,b^{\varphi}]^{-1})^{{\alpha}_3}{\tilde{w}}$. We claim that $[d,b^{\varphi}][c,b^{\varphi}]^{-1}\in {\mathcal{M}_0}^{*}(G_{16})$, and so the result follows. Using Lemma \ref{l2.3} (i) we show that $[db,bc]=1$. Indeed,
\begin{align*}[db,bc]&=[db,c][db,b][db,b,c]\\&=[d,c][d,c,b][b,c][d,b][d,b,b][b,b][[d,b][d,b,b][b,b],c]=f^{-1}f=1.
\end{align*}
Thus $[db,(bc)^{\varphi}]\in {\mathcal{M}_0}^{*}(G_{16})$. Expanding it, we have
\begin{align*}[db,(bc)^{\varphi}]&=[d,c^{\varphi}][d,c^{\varphi},b][b,c^{\varphi}][d,b^{\varphi}][d,b^{\varphi},b][b,b^{\varphi}][d,b^{\varphi},c^{\varphi}]\\&\qquad[d,b^{\varphi},b,c^{\varphi}][b,b^{\varphi},c^{\varphi}]\\&=[d,c^{\varphi}][d,c^{\varphi},b][b,c^{\varphi}][d,b^{\varphi}][d,b^{\varphi},b][b,b^{\varphi}][d,b,c^{\varphi}]\\&\qquad
[d,b,b,c^{\varphi}][b,b,c^{\varphi}].
\end{align*}
It is easy to see that, all of the above commutators except $[b,c^{\varphi}]$ and $[d,b^{\varphi}]$, belong to ${\mathcal{M}_0}^{*}(G_{16})$. So modulo ${\mathcal{M}_0}^{*}(G_{16})$,
$[b,c^{\varphi}][d,b^{\varphi}]=[d,b^{\varphi}][b,c^{\varphi}]=[d,b^{\varphi}][c,b^{\varphi}]^{-1}$, and $[d,b^{\varphi}][c,b^{\varphi}]^{-1}\in {\mathcal{M}_0}^{*}(G_{16})$, as required. Hence $\tilde{B_0}(G_{16})=0$. 
\\
\\
Now let $G\cong G_{132}$. Using Proposition \ref{p2.6}, the group $[G_{132},G_{132}^{\varphi}]$ is generated by $[b,a^{\varphi}] ,[d,b^{\varphi}] , [d,c^{\varphi}] , [e,a^{\varphi}]$, $[f,c^{\varphi}]$ modulo ${\mathcal{M}_0}^{*}(G_{132})$. By using Lemma \ref{l2.3} (vi), we have
$$[[b,a^{\varphi}],[d,b^{\varphi}]]=[[b,a],[d,b]^{\varphi}]=[e,g^{\varphi}]\in {\mathcal{M}_0}^{*}(G_{132}).$$
Similarly,
$$[[b,a^{\varphi}],[d,c^{\varphi}]], [[b,a^{\varphi}],[e,a^{\varphi}]], [[b,a^{\varphi}],[f,c^{\varphi}]], [[d,b^{\varphi}],[d,c^{\varphi}]], [[d,b^{\varphi}],[e,a^{\varphi}]], $$
$$[[d,b^{\varphi}],[f,c^{\varphi}]], [[d,c^{\varphi}],[e,a^{\varphi}]], [[d,c^{\varphi}],[f,c^{\varphi}]], [[e,a^{\varphi}],[f,c^{\varphi}]]\in {\mathcal{M}_0}^{*}(G_{132}).$$
Thus any two elements of the generating set of $[G_{132},G_{132}^{\varphi}]$ are commuting modulo ${\mathcal{M}_0}^{*}(G_{132})$, and each element of $[G_{132},G_{132}^{\varphi}]$ can be written as
$$[b,a^{\varphi}]^{{\alpha}_1} [d,b^{\varphi}]^{{\alpha}_2} [d,c^{\varphi}]^{{\alpha}_3} [e,a^{\varphi}]^{{\alpha}_4}[f,c^{\varphi}]^{{\alpha}_5} {\tilde{w}},$$
where ${\tilde{w}}\in {\mathcal{M}_0}^{*}(G_{132})$, and for all $i$ such that, $1\leq i\leq 5$, ${{\alpha}_i}\in {\Bbb{Z}}$. \\
 Let $w=[b,a^{\varphi}]^{{\alpha}_1} [d,b^{\varphi}]^{{\alpha}_2} [d,c^{\varphi}]^{{\alpha}_3} [e,a^{\varphi}]^{{\alpha}_4}[f,c^{\varphi}]^{{\alpha}_5} {\tilde{w}} \in \mathcal{M}^{*}(G_{132})$. Then\\ $1={\kappa}^{*}(w)=e^{{\alpha}_1}g^{{\alpha}_2}f^{{\alpha}_3}g^{{\alpha}_4}g^{{\alpha}_5}=c^{{\alpha}_1}f^{{\alpha}_{3}}g^{{\alpha}_{2}+{\alpha}_{4}+{\alpha}_{5}}$. Since $c$, $f$ and $g$ are in the polycyclic generating sequence and $\exp(G_{132})=p$,  $c^{{\alpha}_1}=f^{{\alpha}_3}=g^{{\alpha}_{2}+{\alpha}_{4}+{\alpha}_{5}}=1$, so $p$ divides ${\alpha}_1$, ${\alpha}_3$ and $({\alpha}_2+{\alpha}_4+{\alpha}_5)$, respectively. Now Lemmas \ref{l2.3} (ii) and \ref{l2.4} imply that
$$1=[b^p,a^{\varphi}]=[b,a^{\varphi}]^p[b,a^{\varphi},b]^{\binom{p}{2}}.$$
Since $$[b,a^{\varphi},b]^{{\binom{p}{2}}}=[b,a,b^{\varphi}]^{{\binom{p}{2}}}=[e,b^{\varphi}]^{{\binom{p}{2}}}=[e^{{\binom{p}{2}}},b^{\varphi}]=1,$$  
we have $ [b,a^{\varphi}]^{p}=1$. Similarly $[d,c^{\varphi}]^p=1$. Thus \\$w=([d,b^{\varphi}][f,c^{\varphi}]^{-1})^{{\alpha}_2}([e,a^{\alpha}][f,c^{\alpha}]^{-1})^{{\alpha}_4}{\tilde{w}}$ for any $w\in \mathcal{M}^{*}(G_{132})$.\\ If $([d,b^{\varphi}][f,c^{\varphi}]^{-1})$ and $([e,a^{\varphi}][f,c^{\varphi}]^{-1})\in {\mathcal{M}_0}^{*}(G_{132})$, then the result is obtained. \\ We use Lemma \ref{l2.3} (i) to prove that $[dc,bf]=[ec,af]=1$. Expanding the commutators, we have
\begin{align*}[dc,bf]&=[dc,f][dc,b][dc,b,f]\\&=[d,f][d,f,c][c,f][d,b][d,b,c][c,b][[d,b][d,b,c][c,b],f]=g^{-1}g=1.\\
[ec,af]&=[ec,f][ec,a][ec,a,f]\\&=[e,f][e,f,c][c,f][e,a][e,a,c][c,a][[e,a][e,a,c][c,a],f]=g^{-1}g=1.
\end{align*}
Thus $[dc,(bf)^{\varphi}]$ and $[ec,(af)^{\varphi}]\in {\mathcal{M}_0}^{*}(G_{132})$. Expanding $[dc,(bf)^{\varphi}]$, obtain that
\begin{align*}[dc,(bf)^{\varphi}]&=[d,f^{\varphi}][d,f^{\varphi},c][c,f^{\varphi}][d,b^{\varphi}][d,b^{\varphi},c][c,b^{\varphi}][d,b^{\varphi},f^{\varphi}][d,b^{\varphi},c,f^{\varphi}][c,b^{\varphi},f^{\varphi}]\\&=[d,f^{\varphi}][d,f,c^{\varphi}][c,f^{\varphi}][d,b^{\varphi}][d,b,c^{\varphi},b][c,b^{\varphi}][d,b,f^{\varphi}][d,b,c,f^{\varphi}][c,b,f^{\varphi}].
\end{align*}
We can see that all of the above commutators except $[c,f^{\varphi}]$ and $[d,b^{\varphi}]$ belong to ${\mathcal{M}_0}^{*}(G_{132})$. So modulo ${\mathcal{M}_0}^{*}(G_{132})$, $[c,f^{\varphi}][d,b^{\varphi}]=[d,b^{\varphi}][c,f^{\varphi}]=[d,b^{\varphi}][f,c^{\varphi}]^{-1}$, and $[d,b^{\varphi}][f,c^{\varphi}]^{-1}\in {\mathcal{M}_0}^{*}(G_{132})$. Also
\begin{align*}[ec,(af)^{\varphi}]&=[ec,a^{\varphi}f^{\varphi}]=[ec,f^{\varphi}][ec,a^{\varphi}][ec,a^{\varphi},f^{\varphi}]\\&=[e,f^{\varphi}][e,f^{\varphi},c][c,f^{\varphi}][e,a^{\varphi}][e,a^{\varphi},c][c,a^{\varphi}]\\&\qquad[[e,a^{\varphi}][e,a^{\varphi},c][c,a^{\varphi}],f^{\varphi}]\\&=[e,f^{\varphi}][e,f^{\varphi},c][c,f^{\varphi}][e,a^{\varphi}][e,a^{\varphi},c][c,a^{\varphi}]\\&\qquad[e,a^{\varphi},f^{\varphi}][e,a^{\varphi},c,f^{\varphi}][c,a^{\varphi},f^{\varphi}]\\&=[e,f^{\varphi}][e,f,c^{\varphi}][c,f^{\varphi}][e,a^{\varphi}][e,a,c^{\varphi}][c,a^{\varphi}]\\&\qquad[e,a,f^{\varphi}][e,a,c,f^{\varphi}][c,a,f^{\varphi}].
\end{align*}
We can see that, all of the above commutators except $[c,f^{\varphi}]$ and $[e,a^{\varphi}]$ belong to ${\mathcal{M}_0}^{*}(G_{132})$. So modulo ${\mathcal{M}_0}^{*}(G_{132})$, $[c,f^{\varphi}][e,a^{\varphi}]=[e,a^{\varphi}][f,c^{\varphi}]^{-1}$ and $[e,a^{\varphi}][f,c^{\varphi}]^{-1}\in {\mathcal{M}_0}^{*}(G_{132})$, as required. Hence $\tilde{B_0}(G_{132})=0$.

\end{proof}

\section{\bf{Groups of order $p^7$ and exponent $p$ $(p>5)$ with non trivial Bogomolov multiplier}}
In this section, using a similar technique which is used in \cite{4',12}, we will show that the Bogomolov multiplier of some $p$-groups of order $p^7$ and exponent $p$ is non trivial.\\ \\
First in \cite{4'}, Eick and Nickel introduced a useful algorithm for computing a consistent polycyclic presentation of the Schur multiplier and the nonabelian tensor square of a group using consistent polycyclic presentation. Later, Jezernik and Moravec in \cite{12} expanded this method for calculation of the Bogomolov multiplier and curly exterior square of a polycyclic group that has a consistent presentation of the group. \\
Their main tool in this method is to calculate the certain central extensions of a group given by a consistent presentation.
\\  
Let $G$ be a finite polycyclic group defined by a consistent polycyclic presentation $F/R$, where $F$ is the free group on generators $x_i$ $(1\leq i\leq n)$ for some $n$ and following relations
$${x_i}^{e_i}={\prod _{k=i+1}^{n}}{{x_k}^{z_{i,k}}}\ \ \ \ \ ; \ \ \ \ \ 1\leq i\leq n,$$
$$[x_i,x_j]={\prod _{k=i+1}^{n}}{{x_k}^{y_{i,j,k}}}\ \ \ \ \ ; \ \ \ \ \ 1\leq j<i\leq n.$$
Note that, in such a presentation, we omit the trivial commutator relations.
\\
Similar \cite{4',12}, we introduce $l$ new generators $t_1,...,t_l$ (called tails) and we define ${{G}_\oslash^*}$ as the group generated by the generators $x_1,...,x_n,t_1,...,t_l$ and the following relations
$${x_i}^{e_i}={\prod _{k=i+1}^{n}}{{x_k}^{z_{i,k}}}.{t_{l(i)}}\ \ \ \ \ ; \ \ \ \ \ 1\leq i\leq n,$$
$$[x_i,x_j]={\prod _{k=i+1}^{n}}{{x_k}^{y_{i,j,k}}}.{t_{l(i,j)}}\ \ \ \ \ ; \ \ \ \ \ 1\leq j<i\leq n.$$
The next lemma states some facts about ${{G}_\oslash^*}$.

\begin{lemma}\cite[Lemma 1]{4'}\label{l4}
Let $G$ be defined by the consistent polycyclic presentation $F/R$ as above. Using the above notation we obtain the following.
\renewcommand {\labelenumi}{(\roman{enumi})} 
\begin{enumerate}
\item{${{G}_\oslash^*}\cong F/[R,F] \ \ , \ \ T\cong R/[R,F]  \ \ , \ \ {{{G}_\oslash^*}}/T\cong F/R\cong G$,}
\item{${{G}_\oslash^*}$ is defined by a polycyclic presentation.}
\end{enumerate}
Where $T:=<t_1,...,t_l>$.
\end{lemma}
Therefore ${{G}_\oslash^*}$ is a central extension of $T$ by $G$, but the given relations may determine an inconsistent presentation for ${{G}_\oslash^*}$. Now, by using following consistency relations, we introduce consistency relations between the tails
$${x_k}({x_j}{x_i})=({x_k}{x_j}){x_i}\ \ \ \ \ ; \ \ \ \ \ k>j>i,$$
$$({x_j}^{e_j}){x_i}={{x_j}^{{e_j}-1}}({x_j}{x_i})\ \ \ \ \ ; \ \ \ \ \ j>i,$$
$${x_j}({x_i}^{e_i})=({x_j}{x_i}){{x_i}^{{e_i}-1}}\ \ \ \ \ ; \ \ \ \ \ j>i,$$
$$({x_i}^{e_i}){x_i}={x_i}({x_i}^{e_i})\ \ \ \ \ \ \ \ \ \textsl{for all i}.$$
In addition to the consistency relations, we want to evaluate the commutators $[x,y]$ in the extension ${{G}_\oslash^*}$ with the elements $x$, $y$ commuting in $G$, which themselves cause the creation of some new tail relations. In fact this step is the same as determining the subgroup ${\mathcal{M}_0}(G)\cong {<K(F)\cap R>}/{[R,F]}$ of the Schur multiplier $\mathcal{M}(G)\cong {(R\cap F')}/{<K(F)\cap R>}$, that was mentioned in the first two sections.
\\
Let ${{G}_\circ^*}$ be the group obtained by factoring ${{G}_\oslash^*}$ by these extra relations. Using the Gaussion elimination method, we produce a generating set for all relations between the tails and collect in the matrix $T$. Then we introduce a new basis for the tails $($say $t_l^*$$)$, by using Smith normal form $S=PTQ$ of $T$ that will be obtained by two invertible matrices $Q$ and $P$. The abelian invariants of the group generated by the tails are identified as the elementary divisors of $T$. Finally, the Bogomolov multiplier of $G$ is recognized as the torsion subgroup of $<t_l^* \ | \ 1\leq l\leq m>$ inward ${G}_\circ^*$, $($for more information you can see \cite{4',12}$)$.
\\
\\ 
Note that the theoretical history of this method being the following proposition which Jezernik and Moravec have proved it in their article.
\begin{proposition}\cite[Proposition 2.1]{12}\label{p3.4'}
Let $G$ be a finite group presented by $G =
F/R$ with $F$ free of rank $n$. Denote by $K(F)$ the set of commutators in $F$. Then $\tilde{B_0}(G)$ is isomorphic to the torsion subgroup of $R/(K(F) \cap R)$ and the torsion free factor $R/([F, F] \cap R)$ is free abelian of rank $n$. Moreover, every complement $C$ to $\tilde{B_0}(G)$ in $R/(K(F) \cap R)$ yields a commutativity preserving central extension of $\tilde{B_0}(G)$ by $G$. 
\end{proposition}
Now by using this method, we show that following groups have nontrivial Bogomolov multiplier. 
\begin{proposition}\label{p3.4}
All groups, except the groups of the Propositions \ref{p3.1}, \ref{p3.2} and \ref{p3.3}, have non trivial Bogomolov multiplier.

\end{proposition}
\begin{proof}
We state the proof in detail for $G_9$, the remaining cases can be proved in a similar way. The group $G_9$ has a following polycyclic presention
$$G_{9}= <a,b,c,d,e,f,g \ | \  [b,a]=c , [c,a]=d , [c,b]=[d,a]=e>.$$

We add $11$ tails to the presentation to make a quotient of the universal central extension of the system: 

$$a^p=t_1, b^p=t_2, c^p=t_3, d^p=t_4, e^p=t_5, f^p=t_6, g^p=t_7,$$
$$[b,a]=ct_8, [c,a]=dt_9, [c,b]=et_{10}, [d,a]=et_{11}.$$

So we have a new group that generated by $a,b,c,d,e,f,g$ and $t_i$, where $1\leq i \leq 11$. On the other hand, by using consistency relations, we have following relations between the tails:
$${{t_8}^p}{t_3}=1 \ \ , \ \ {{t_9}^p}{t_4}=1 \ \ , \ \ {{t_{10}}^p}{t_5}=1 \ \ , \ \ {{t_{11}}^p}{t_5}=1.$$

Now we collect all the coefficients of these relationships in a matrix and using elementary row (column) operations, we turn it into the upper triangular matrix $T$.

\begin{center}

T=\ \begin{blockarray}{ccccccccccc}
\matindex{$t _{1} $} & \matindex{$t _{2} $} & \matindex{$ t _{3} $} &\matindex{$ t _{4} $} &\matindex{$ t _{5} $} &\matindex{$ t _{6} $} &\matindex{$ t _{7} $} &\matindex{$ t _{8} $} &\matindex{$ t _{9} $} &\matindex{$ t _{10} $} &\matindex{$ t _{11} $} &\\
\begin{block}{(ccccccccccc)}
0&0&1&0&0&0&0&p&0&0&0\\
0&0&0&1&0&0&0&0&p&0&0\\
0&0&0&0&1&0&0&0&0&p&0\\
0&0&0&0&1&0&0&0&0&0&p\\
\end{block}
 \end{blockarray}
 
\end{center}
 
 Then we change basis for the tails (say $t_l^*$), by using Smith normal form $S=PTQ$ which is obtained by the  following two invertible matrices $P$ and $Q$.
 
\begin{center}
P=\ \begin{blockarray}{cccc}
\begin{block}{(cccc)}
1&0&0&0\\
0&1&0&0\\
0&0&1&0\\
0&0&-1&1\\
\end{block}
 \end{blockarray}
 \end{center}
 
 \begin{center}
Q=\ \begin{blockarray}{ccccccccccc}
\begin{block}{(ccccccccccc)}
1&0&0&0&0&0&0&0&0&0&0\\
0&1&0&0&0&0&0&0&0&0&0\\
0&0&1&0&0&0&0&-p&0&0&0\\
0&0&0&1&0&0&0&0&-p&0&0\\
0&0&0&0&1&0&0&0&0&-p&0\\
0&0&0&0&0&1&0&0&0&0&0\\
0&0&0&0&0&0&1&0&0&0&0\\
0&0&0&0&0&0&0&1&0&0&0\\
0&0&0&0&0&0&0&0&1&0&0\\
0&0&0&0&0&0&0&0&0&1&0\\
0&0&0&0&0&0&0&0&0&1&1\\
\end{block}
 \end{blockarray}
 \end{center}

\begin{center}
S=\ \begin{blockarray}{ccccccccccc}
\matindex{$t _{1}^* $} & \matindex{$t _{2}^* $} & \matindex{$ t _{3}^* $} &\matindex{$ t _{4}^* $} &\matindex{$ t _{5}^* $} &\matindex{$ t _{6}^* $} &\matindex{$ t _{7}^* $} &\matindex{$ t _{8}^* $} &\matindex{$ t _{9}^* $} &\matindex{$ t _{10}^* $} &\matindex{$ t _{11}^* $} &\\
\begin{block}{(ccccccccccc)}
0&0&1&0&0&0&0&0&0&0&0\\
0&0&0&1&0&0&0&0&0&0&0\\
0&0&0&0&1&0&0&0&0&0&0\\
0&0&0&0&0&0&0&0&0&0&p\\
\end{block}
 \end{blockarray}
 .
 \end{center}
 
Therefore the non trivial elementary divisors of the Smith
normal form of $T$ are $1, 1, 1, p$. The element $t_{11}^*$ is corresponding to the divisor that is greater than $1$. Thus, $\tilde{B_0}(G)\cong <t_{11}^* \ | \ ({t_{11}^*})^p=1>$.
We know the Bogomolov multiplier has only nonuniversal commutator relations as generators. So, we must discard the new tails $t_{i}^*$ that their corresponding elementary divisors are trivial or $1$. After converting the situation back to the principal tails $t_i$, we have $t_{11}^*=t_{11}$ and the other tails $t_i$ are trivial. Thus we have a commutativity preserving central extension (CP extension) of the tails subgroup by $G_9$, that has following presentation
$$<a,b,c,d,e,f,g,{t_{11}^*} \ | \ a^p=b^p=c^p=d^p=e^p=f^p=g^p=(t_{11}^*)^p=1 \ , \ $$
\[ [b,a]=c , [c,a]=d , [c,b]=e , [d,a]=e{t_{11}^*}>.\]
On the other hand its derived subgroup is isomorphic to ${G_9}\curlywedge {G_9}$ (see \cite{12}). Also, we know $t_{11}^*=[c,b][d,a]^{-1}$. Finally, we have 
\[\tilde{B_0}(G_9)\cong <[c,b][d,a]^{-1} \ | \ ([c,b][d,a]^{-1})^p=1>\cong {\Bbb{Z}}_p.\]

\end{proof}

\section{\bf{The Bogomolov multiplier of groups of order $3^7$ and exponent $3$}}
According to the Wilkinson's classification in \cite{31}, all groups of exponent $3$ and order $3^7$ are given by the $16$ groups in the list with the following numbers
$$\{1, 2, 4, 6, 10, 11, 12, 13, 35, 36, 37, 38, 39, 43, 44, 110\}.$$
The $17$th such group is the Burnside group on three generators. with the commutator relations
$$G_{17}={[b,a]=d , [c,a]=e , [c,b]=f , [d,c]=g , [b,e]=g , [f,a]=g}.$$
Similar to the previous content, $\tilde{B_0}(G_i)=0$, where $i\neq 38$. Also, we can show that, $G$ have trivial Bogomolov multiplier. 
\section{\bf{The Bogomolov multiplier of groups of order $5^7$ and exponent $5$}}
Here also according to the Wilkinson's classification \cite{31}, all groups of exponent $5$ and order $5^7$ are given by the following $153$ groups in a list containing the following numbers 
\\
$\{1, ... , 29, 35, ... , 82, 84, ... , 102, 110, ... , 120, 126-n (n=0,...,4), 130, ... , 147,$
\\
$148-n (n=2,3,4), 149, 150, ... , 152, 153-n (n=0,...,4), 154, 155-n (n= 2, 3, 4),$
\\
$156, 157, 158-n (n= 1, 2), 159, 160, 161-kn (k= 1,n = 4)\}.$
\\
\\
Therefore $\tilde{B_0}(G_i)=0$, where\\
$i\in \{1,2,3,4,5,6,7,8,10,11,12,15,16,17,20,21,22,24,25,26,28,35,36,37,39,40,$
\\
$41,43,45,46,47,48,49,50,52,54,55,57,59,60,62,64,70,71,73,74,75,76,81,85,$
\\
$89,91,95,100,111,112,114,117,118,131,132,133,140,141,142,143,144,149,152\}$.

\end{document}